\documentclass[11pt]{amsart}
\usepackage{amsmath,mathrsfs,amssymb,comment,enumerate,url,tikz-cd}
\input{xypic}
\xyoption{all}

\usepackage{xcolor} 
\colorlet{mdtRed}{red!50!black}
\definecolor{dblue}{rgb}{0,0,.6}
\usepackage[colorlinks]{hyperref}
\hypersetup{linkcolor=blue,citecolor=dblue,filecolor=dullmagenta,urlcolor=mdtRed}

\numberwithin{equation}{section}
\newtheorem{theorem}[equation]{Theorem}
\newtheorem{corollary}[equation]{Corollary}
\newtheorem{lemma}[equation]{Lemma}
\newtheorem{proposition}[equation]{Proposition}

\newtheorem{notation}[equation]{Notation}
\newtheorem*{theorem*}{Theorem}
\newtheorem*{corollary*}{Corollary}
\newtheorem*{proposition*}{Proposition}

\theoremstyle{remark}

\newcommand{\mf}[1]{\mathfrak{#1}}
\newcommand{\ms}[1]{\mathscr{#1}}
\newcommand{\mb}[1]{\mathbb{#1}}
\newcommand{\mc}[1]{\mathcal{#1}}
\begin{document}

\title{Automorphisms of relative Quot schemes}

\author{Chandranandan Gangopadhyay}

\address{School of Mathematics, Tata Institute of Fundamental Research,
Homi Bhabha Road, Mumbai 400005, India}

\email{chandrag@math.tifr.res.in}



\begin{abstract}
Let $k$ be an algebraically closed field of
characteristic zero. 
Let $S$ be a smooth projective variety over $k$
and let $p_S:X\rightarrow S$ be 
a family of smooth projective curves over $S$. 
Let $E$ be a vector bundle over $X$. For $s\in S$
let $X_s$ be the fibre of $p_S$ over $s$ 
and let $E_s$ be the restriction
of $E$ to $X_s$. Fix $d\geq 1$.
Let $\mc Q(E,d)\to S$ be the relative Quot scheme parameterizing
torsion quotients of $E_s$ over $X_s$ of degree $d$ for all $s\in S$. 
In this article
we compute the identity component of relative automorphism group scheme
which parameterizes automorphisms of $\mc Q(E,d)$ over $S$.
\end{abstract}

\if
and let $\{X_s\}_{s\in S}$ be a family of smooth
projective curves over $S$. Let $\{E_s\}_{s\in S}$
be a family of vector bundles over
 over $S$. 
Fix $d\geq 1$. Let $\mathcal{Q}(E,d)\to S$ 
be the relative quot scheme associted to the family $X\to S$ 
parametrizing torsion quotients 
of $E$  of degree $d$. In this article, we compute the 
the identity component of relative automorphism group scheme
of $\mc Q(E,d)$ over $S$.
 
 Then we show that if $r\geq 3$, then the identity component of the group of automorphisms of $\mathcal{Q}(E,d)$ over $S$ is isomorphic to the identity component of the group of automorphisms of $\mathbb{P}(E)$ over $S$. We also show that under additional hypotheses, the same statement is true in the case when r=2. As a corollary, the identity component of the automorphism group of flag scheme of filtrations of torsion quotients of $\mathcal{O}^{r}_{C}$, where $r\geq 3$ and $C$ a smooth projective curve of genus $\geq 2$ is computed. 
\fi

\maketitle

\section{Introduction}

Let $k$ be an algebraically closed field of characteristic zero. 
Let $Y\rightarrow S$ be a smooth morphism 
between two projective varieties over $k$. 
Associated to this morphism we have 
the automorphism group scheme ${\rm Aut}(Y/S)$ 
which parameterizes automorphisms of $Y$ over $S$. 
Let us denote identity component of 
${\rm Aut}(Y/S)$ by ${\rm Aut}^{o}(Y/S)$. It is known that 
${\rm Aut}^{o}(Y/S)$ is an algebraic group and if
 $\mathcal{T}_{Y/S}$ is the relative tangent bundle,  
then 
${\rm Lie}({\rm Aut}^{o}(Y/S))=H^{0}(Y,\mathcal{T}_{Y/S})$ \cite[Theorem 3.7]{MO}, \cite[Theorem 2.3]{Bri18}.
We refer to \cite{Bri14}, \cite{Bri18} for other properties of this group scheme.

We refer to \cite[Section 2]{HL} for definitions and properties of 
Quot Schemes in general. The Quot Scheme which we will study in this 
article can be defined in the following manner.
Let $p_{S}:X\rightarrow S$ be a family of smooth projective curves 
over an algebraically closed field $k$ of characteristic zero. 
Assume $X$ and $S$ are smooth projective varieties. 
Let $E$ be a vector bundle over $X$ of rank $r$. For a closed point 
$s\in S$
let $X_s$ be the fibre of $p_S$ over $s$ 
and let $E_s$ be the restriction
of $E$ to $X_s$.
Fix $d \geq 1$. Then associated to the morphism $p_S$ and 
the vector bundle $E$ we have the relative Quot scheme 
 $\pi_{S}:\mathcal{Q}(E,d)\rightarrow S$  
whose closed points correspond to quotients 
$E_s \rightarrow B_{d}$, $\forall s\in S$ 
where $B_{d}$ is a torsion sheaf of degree $d$ 
over the smooth projective curve 
$X_{s}$ \cite[Theorem 2.2.4]{HL}. 
 It is known that $\mc Q(E,d)$ is a smooth projective variety \cite[Proposition 2.2.8]{HL}. These schemes have been studied extensively. We refer the reader to \cite{BGL}, \cite{BDW}, \cite{BDH} for other properties of this scheme. 
In this article we compute the group scheme ${\rm Aut}^o(\mc Q(E,d)/S)$.
We recall that in the case when $S$ is a point and $E$ the 
trivial bundle of rank $r$ this group scheme was computed in 
\cite{BDH}. In \cite{BM}  this group scheme was computed in another special case.
We refer to Corollary \ref{corollary-g greater than 1} and Corollary \ref{automorphism of generalised quot scheme}
where these results are stated explicitly.

Over $X$ we fix a certain ample line bundle $\mc M$ (this line bundle is defined just before Lemma \ref{extending section-2}).
Then the main theorem of this article is
\begin{theorem*}[Theorem \ref{main theorem}]
 Suppose either $r:={\rm rank}~E\geq 3$ or $r=2$, $E$ is semistable with respect to $\mc M$ and genus of $X_s$ is $\geq 2$ for $s\in S$. In both of these cases we have isomorphisms
\begin{enumerate}
\item ${\rm Aut}^{o}(\mathcal{Q}(E,d))\cong {\rm Aut}^{o}(\mathbb{P}(E)/S)\,.$
\item $H^{0}(\mathbb{P}(E),\mathcal{T}_{\mathbb{P}(E)/S})\cong H^{0}(\mathcal{Q}(E,d),\mathcal{T}_{\mathcal{Q}(E,d)/S})\,.$
\end{enumerate}
\end{theorem*}
As consequences of Theorem \ref{main theorem} we deduce the results of 
\cite{BDH} and \cite{BM} as Corollary \ref{corollary-g greater than 1} and Corollary \ref{automorphism of generalised quot scheme}. We also compute the identity component of the automorphism group scheme of the flag scheme parameterizing chains of torsion quotients of trivial bundle over a smooth projective curve(Corollary \ref{automorphisms of flag schemes}). We refer to Section \ref{Applications} for more details. 
\section{Main Theorem}\label{section-Main theorem}
Let us denote the projection $X\times_{S} \mathcal{Q}(E,d) \rightarrow X$ by $\pi_{X}$ and the projection $X\times_{S} \mathcal{Q}(E,d) \rightarrow \mathcal{Q}(E,d)$ by $p_{\mathcal{Q}}$ i.e. we have the following diagram:
\[
\begin{tikzcd}
X\times_{S} \mathcal{Q}(E,d) \arrow[r, "\pi_{X}"] \arrow[d,"p_{\mathcal{Q}}"] & X \arrow[d, "p_{S}"] \\
\mathcal{Q}(E,d) \arrow[r ,"\pi_{S}"] & S
\end{tikzcd}
\]
We denote the universal quotient on $X\times_{S}\mathcal{Q}(E,d)$ by
$$\pi_{\mathcal{Q}}^{*}E\rightarrow \mathcal{B}\rightarrow 0\,.$$
\begin{lemma}\label{inclusion of aut of quot in aut of proj} 
We have a closed immersion of algebraic groups
$${\rm Aut}^{o}(\mathbb{P}(E)/S) \hookrightarrow {\rm Aut}^{o}(\mathcal{Q}(E,d)/S)\,.$$
\end{lemma}
\begin{proof} 
By \cite[Corollary 2.2]{MO} any automorphism $g\in \text{Aut}^{o}(\mathbb{P}(E)/S)$ descends to an automorphism $h\in \text{Aut}^{o}(X/S)$. Therefore we have the following diagram:
\[
\begin{tikzcd}
\mathbb{P}(E)\cong \mathbb{P}((g_1)^{*}E) \arrow[r,"g"] \arrow[d,"p"] & \mathbb{P}(E) \arrow[d,"p"] \\
X \arrow[r,"h"] & X
\end{tikzcd}
\]
Then $E\cong h^{*}  E\otimes p^{*}L$ 
for some line bundle $L$ on $X$. 
Let us denote this isomorphism of bundles by $\Psi_{g}$.
\if
For $s\in S$ we define $\Psi_{g,s}:=\Psi_g|_{X_s}$. 
Then $\Psi_{g}$ induces an automorphism of 
$\mathcal{Q}(E,d)$ by sending 
$$[E_s \rightarrow B_{d}\rightarrow 0]
\to 
[E_s \xrightarrow{\Psi_{g,s}} (g')^{*}(E_s \rightarrow B_{d}\rightarrow 0)\otimes \mathcal{L}_s]$$
Hence we have a homomorphism $\text{Aut}^{o}(\mathbb{P}(E)/S) \rightarrow \text{Aut}^{o}(\mathcal{Q}(E,d)/S)$ and clearly it is injective.

Let $g\in \text{Aut}^{o}(\mathbb{P}(E)/S)$.
Note that by \cite[Cor 2.2]{7}, any automorphism $g\in \text{Aut}^{o}(\mathbb{P}(E)/S)$ descends to an automorphism $g'\in \text{Aut}^{o}(X/S)$. Therefore we have the following diagram:\\

\hspace{2cm}\begin{tikzcd}
\mathbb{P}(E)\cong \mathbb{P}((g')^{*}E) \arrow[r,"g"] \arrow[d,"f"] & \mathbb{P}(E) \arrow[d,"f"] \\
X \arrow[r,"g'"] & X
\end{tikzcd} \\
Then, $E\cong (g')^{*}  E\bigotimes p^{*}\mathcal{L}$ for some line bundle $\mathcal{L}$ on $X$. Let us denote this isomorphism of bundles by $\Psi_{g}$.\\

Let us denote the projection $X\times_{S} \mathcal{Q}(E,d) \rightarrow X$ by $\pi_{X}$ and the projection $X\times_{S} \mathcal{Q}(E,d) \rightarrow \mathcal{Q}(E,d)$ by $p_{\mathcal{Q}}$ i.e. we have the following diagram:
\[
\begin{tikzcd}
X\times_{S} \mathcal{Q}(E,d) \arrow[r, "\pi_{X}"] \arrow[d,"p_{\mathcal{Q}}"] & X \arrow[d, "p_{S}"] \\
\mathcal{Q}(E,d) \arrow[r ,"\pi_{S}"] & S
\end{tikzcd}
\]
and we have the universal quotient $\pi_{\mathcal{Q}}^{*}E\rightarrow \mathcal{B}\rightarrow 0$ on $X\times_{S}\mathcal{Q}(E,d)$.
\fi
On $X\times_S \mc Q(E,d)$ consider the quotient 
$$\pi^{*}_{\mathcal{Q}}E\cong \pi^{*}_{\mathcal{Q}}h^{*}E\otimes \pi^{*}_{\mathcal{Q}}L\rightarrow (h\times id)^{*}\mathcal{B}\otimes \pi^{*}_{\mathcal{Q}}L$$ 
where the first isomorphism is induced from $\Psi_{g}$ and the second morphism is the pullback of the universal quotient $\pi^{*}_{\mathcal{Q}}E\rightarrow \mathcal{B}\rightarrow 0$ under the map $(h\times {id})$ tensored with $\pi^{*}_{\mathcal{Q}}L$. This gives a quotient of $\pi^{*}_{\mathcal{Q}}E$ over $X\times_{S}\mathcal{Q}(E,d)$ and by the universal property of Quot schemes this induces an automorphism of $\mathcal{Q}(E,d)$. Hence, we have a homomorphism 
$$\text{Aut}^{o}(\mathbb{P}(E)/S) \rightarrow \text{Aut}^{o}(\mathcal{Q}(E,d)/S)\,.$$ Next we show that this homomorphism is injective. At the level of closed points the above automorphism of $\mathcal{Q}(E,d)$ induced by $g\in {\rm Aut}^o(\mb P(E)/S)$ is given by 
$$[E_s \rightarrow B_{d}\rightarrow 0]
\to 
[E_s \xrightarrow{\Psi_{g,s}} h^{*}(E_s \rightarrow B_{d}\rightarrow 0)\otimes L_s]$$
Suppose $g\in \text{Aut}^{o}(\mathbb{P}(E)/S)$ induces the identity automorphism on $\mathcal{Q}(E,d)$. We will show that $g=\text{id}$.
First we show that $h={\rm id}$. Fix $x\in X$ and let $p_S(x)=s\in S$. Consider any quotient 
$$E_s\to \mathcal{O}_{X_s,x}/\mf{m}_{X_s,x}^{d}$$ 
where $\mathcal{O}_{X_s,x}$ is the local ring of $X_s$ at $x$ and $\mf m_{X_s,x}$ is its maximal ideal. Then under the automorphism induced by $g$ the image of this quotient is of the form
$$E_s\to  \mathcal{O}_{X_s,h(x)}/\mf{m}_{X_s,h(x)}^{d} $$
Hence if $g$ induces the identity automorphism of $\mc Q(E,d)$ then 
$h={\rm id}$.
Next we show that $g=\text{id}$. Let $v\in \mb P(E)$, let $p(v)=x$ and $p_S(x)=s$. Then $v$ corresponds to a quotient of vector spaces
$E|_{x}\xrightarrow{v} k$. Since, $h=\text{id}$, $p(g(v))=x$. is a quotient of the form $E|_{x}\rightarrow k$.
Let us fix $d-1$  degree $1$ quotients  $E\xrightarrow{v_i} k_{x_i}$ for $1\leq i \leq d-1$ such that all $x_i,x$ are distinct and $p_S(x_{i})=p_S(x)=s$. Then  the summation of these quotients gives us a point in $\mc Q(E,d)$ 
$$E_s \to E|_{x} \oplus  \bigoplus_{i=1}^{d-1}E|_{x_{i}} \to k_{x} \oplus \bigoplus_{i=1}^{d-1} k_{x_{i}}\,.$$ 
Note that each of the quotients $E|_{x}\xrightarrow{h} k$ and $E|_{x_{i}}\rightarrow k$ can be recovered from the above degree $d$ quotient simply by restricting this quotient to the points $x$ and $x_i$ respectively. By assumption the automorphism induced by $g$ is identity. Therefore applying the automorphism induced by $g$ and restricting it to $x$, we get that $g(v)=v$. This completes the proof of injectivity.
\end{proof}

\begin{corollary}\label{inclusion of lie algebras}
We have an inclusion of lie algebras
$$H^{0}(\mathbb{P}(E),\mathcal{T}_{\mathbb{P}(E)/S})\hookrightarrow H^{0}(\mathcal{Q}(E,d),\mathcal{T}_{\mathcal{Q}(E,d)/S})\,.$$
\end{corollary}
\begin{proof}
This follows  from 
Lemma \ref{inclusion of aut of quot in aut of proj} 
and \cite[Theorem 3.7]{MO}.
\end{proof}
\if
We will prove that the inclusion in 
Lemma \ref{inclusion of aut of quot in aut of proj} 
and Corollary \ref{inclusion of lie algebras} 
are actually isomorphisms. 
We will show this first for
the case $S=\{{\rm pt}\}$.  
\fi
Let  $\mathcal{Z}$ be the fibered product 
of $d$ copies of $\mathbb{P}(E)$ over $S$. 
We will construct a rational map 
$\Phi:\mc Z \dashrightarrow \mc Q(E,d)$.
Note that this map was already constructed 
in \cite[Section 2]{G18} in the special case when $S=\{{\rm pt}\}$ and $E=\mc O^r_X$.
First we set some notations.
\begin{notation}~~\label{notation}
\begin{enumerate}
\item Let $p:\mathbb{P}(E)\rightarrow X$ be the projection. 
\item Let $p_{i}:\mc Z \to \mb P(E)$ be the $i$-th projection. 
\item For, $i\neq j$, let $\Delta_{i,j}\hookrightarrow \mc Z$ be the closed subscheme  given by the equation $p_{i}=p_{j}$. 
\item For $i,j$ distinct let $\Delta_{i,j,X}\hookrightarrow \mc Z$ be the closed subscheme given by the equation $p\circ p_{i}=p\circ p_{j}$.
\item For $i,j,k$ all distinct, let $\Delta_{i,j,k,X}\hookrightarrow \mc Z$ be the closed subscheme given by the equation $p\circ p_{i}=p\circ p_{j}=p \circ p_{k}$.
\item Let $\pi_{1}:X\times_S \mc Z\to X$ and $\pi_{2}:X\times_S \mc Z\to \mc Z$ be the first and second projections respectively. 
\item Let $p_{i}\circ \pi_{2}:X\times_S \mc Z\to \mb P(E)$ be denoted by $\pi_{2,i}$.
\item Let $\Delta_{i}\hookrightarrow X\times_{S} \mathcal{Z}$ be the closed subscheme given by the equation $\pi_{1}=p\circ \pi_{2,i}$.
\end{enumerate}
\end{notation}
We define an open set 
$$\mathcal{U}:=\mathcal{Z} \setminus (\bigcup\limits_{i,j} \Delta_{i,j} \bigcup\bigcup\limits_{i,j,k } \Delta_{i,j,k,X} )\,.$$
Consider the following compostion of morphisms over $X\times_{S}\mathcal{Z}$
$$
q:\pi^{*}_{1}E \to \bigoplus\limits^d_{i=1} \pi^{*}_{1}E|_{\Delta_{i}} \cong \bigoplus\limits^d_{i=1} \pi^{*}_{2,i}E|_{\Delta_{i}}  \to   \bigoplus\limits_{i=1}^{d} \pi^{*}_{2,i}\mathcal{O}(1)|_{\Delta_{i}}\,.
$$
Let $u \in \mathcal{U}$. Then $q|_{X\times_S u}$ is the morphism 
$$E\to \bigoplus\limits^d_{i=1} E|_{p\circ p_i(u)}\to \bigoplus\limits^d_{i=1} k_{p\circ p_i(u)}$$
where the map $E|_{p\circ p_i(u)}\to k_{p\circ p_i(u)}$ is given by $p_i(u)\in \mb P(E)$. Since $u\in \mc U$  for any $1\leq i \leq d$, there can exist atmost one pair $j\neq i$ such that $p\circ p_{i}(u)=p\circ p_{j}(u)$, and for such a pair $(i,j)$, $p_{i}(u)\neq p_{j}(u)$. Hence $q|_{X\times_S u}$ is a surjection. Therefore $q|_{X\times_S \mc U}$ is a surjection. By universal property of $\mathcal{Q}(E,d)$ the surjection $q|_{X\times_S \mc U}$ induces a map 
$$\Phi:\mathcal{U} \rightarrow \mathcal{Q}(E,d)\,.$$ 
\if
 $(\bigoplus\limits_{i=1}^{d} \pi^{*}_{2,i}\mathcal{O}(1)|_{\Delta_{i}})|_{X\times_{S}\{u\}}=\bigoplus\limits_{i=1}^{d} k_{p\circ p_{i}(u)}$, and the surjection $E|_{X\times_{S}\{u\}}\rightarrow k_{p\circ p_{i}(u)}$ is given by the element $p_{i}(u)$. \\
By definition of $\mathcal{U}$, for any $1\leq i \leq d$, there can exist atmost one $j\neq i$ such that $p\circ p_{i}(u)=p\circ p_{j}(u)$, and for such a pair $(i,j)$, $p_{i}(u)\neq p_{j}(u)$, i.e. the homomorphism $E|_{X\times_{S}\{u\}}\rightarrow k_{p\circ p_{i}(u)}\bigoplus k_{p\circ p_{j}(u)}$ is a surjection.\\
Therefore, $q|_{X\times_{S} \{u\}}$ is surjective $\forall u\in \mathcal{U}$, and clearly $(\bigoplus\limits_{i=1}^{d} \pi^{*}_{2,i}\mathcal{O}(1)|_{\Delta_{i}})|_{X\times_{S}\{u\}}=\bigoplus\limits_{i=1}^{d} k_{p\circ p_{i}(u)}$ has length $d$.
\fi
Then we prove the following proposition
\begin{proposition*}[Propositon \ref{global sections of pullback of the tangent bundle}]
Suppose either $r\geq 3$ or $r=2$, $E$ is semistable with respect to $\mc M$ and genus of $C$ is $\geq 2$. In both of these cases we have  an isomorphism
$$H^{0}(\mathcal{U},\Phi^{*}\mathcal{T}_{\mathcal{Q}(E,d)/S})=\bigoplus\limits_{i=1}^{d}H^{0}(\mathbb{P}(E),\mathcal{T}_{\mathbb{P}(E)/S})\,.$$
\end{proposition*} 

To prove Proposition \ref{global sections of pullback of the tangent bundle}
we need a few lemmas. We define $\mc F(E,d):={\rm ker}~q$.
 
\begin{lemma}\label{pullback of the tangent bundle}
We have an isomorphism of vector bundles
$$\Phi^*\mc T_{\mc Q(E,d)/S}\cong \bigoplus\limits^d_{i=1}(\pi_2)_*\ms Hom(\mc F(E,d),\pi^{*}_{2,i}\mathcal{O}(1)|_{\Delta_{i}})\vert_{\mc U}\,.$$
\end{lemma}

\begin{proof}
\if
Let us denote the projection $X\times_{S} \mathcal{Q}(E,d) \rightarrow X$ by $\pi_{X}$ and the projection $X\times_{S} \mathcal{Q}(E,d) \rightarrow \mathcal{Q}(E,d)$ by $p_{\mathcal{Q}}$ i.e. we have the following diagram:
\[
\begin{tikzcd}
X\times_{S} \mathcal{Q}(E,d) \arrow[r, "\pi_{X}"] \arrow[d,"p_{\mathcal{Q}}"] & X \arrow[d, "p_{S}"] \\
\mathcal{Q}(E,d) \arrow[r ,"\pi_{S}"] & S
\end{tikzcd}
\]
\fi
Over $X\times_{S} \mathcal{Q}(E,d)$, we have the universal exact sequence:
$$0 \to \mc A(E,d) \to \pi^*_X E\to \mc B(E,d)\to 0 \,.$$
Then it is known that  $\mathcal{A}(E,d)$ is a vector bundle of rank $r$ \cite[Lemma 2.2]{G18} and by \cite[Proposition 2.2.7]{HL} we have
$$\mathcal{T}_{\mathcal{Q}(E,d)/S}=(p_{\mathcal{Q}})_{*}\mathcal{H}om(\mathcal{A}(E,d),\mathcal{B}(E,d))\,.$$
Consider the following diagram:
\[
\begin{tikzcd}
X\times_{S} \mathcal{U} \arrow[r,"id_{X}\times \Phi"] \arrow[d,"\pi_{1}"] & X\times_{S} \mathcal{Q}(E,d) \arrow[d,"p_{\mathcal{Q}}"] \\
\mathcal{U} \arrow[r,"\Phi"] & \mathcal{Q}(E,d)
\end{tikzcd}
\]
By Grauert's theorem \cite[Corollary 12.9]{Har}, We get that
\begin{align*}
\Phi^{*}\mathcal{T}_{\mathcal{Q}(E,d)/S} = & (\Phi)^{*}(p_{\mathcal{Q}})_{*}\mathcal{H}om(\mathcal{A}(E,d),\mathcal{B}(E,d))                        \\
                                         \cong  & (\pi_{1})_{*}(id_{X}\times \Phi)^{*}\mathcal{H}om(\mathcal{A}(E,d),\mathcal{B}(E,d)) \,.
\end{align*}
Since $\mathcal{A}(E,d)$ is a vector bundle, we have
\begin{align*}
& (id_{X}\times_{S}\Phi)^{*}\ms{H}om(\mathcal{A}(E,d),\mathcal{B}(E,d))     \\
& =\ms{H}om((id_{X}\times_{S}\Phi)^{*}\mathcal{A}(E,d),(id_{X}\times_{S}\Phi)^{*}\mathcal{B}(E,d))\,.                                                                \\
\end{align*}
By the definition of the map $\Phi$ we have 
$$(id_{X}\times_{S}\Phi)^{*}\mathcal{B}(E,d)\cong (\bigoplus\limits_{i=1}^{d} \pi^{*}_{2,i}\mathcal{O}(1)|_{\Delta_{i}})|_{X\times_{S}\mathcal{U}}\,.$$ 
Also 
$$\Phi^{*}\mathcal{A}(E,d)\cong \mathcal{F}(E,d)|_{X\times_{S} \mathcal{U}}\,.$$ 
since by \cite[Lemma 2.2]{G18} $\mathcal{F}(E,d)|_{X\times_{S}\mathcal{U}}$ is again a vector bundle of rank $r$ and there exists a surjection $\Phi^{*}\mathcal{A}(E,d)\twoheadrightarrow \mathcal{F}(E,d)|_{X\times_{S} \mathcal{U}}$.
This completes the proof of the lemma.
\if 
 Therefore, 
 $$\Phi^{*}\mathcal{T}_{\mathcal{Q}(E,d)/S}=\mathcal{H}om(\mathcal{F}(E,d),\bigoplus\limits_{i=1}^{d} \pi^{*}_{2,i}\mathcal{O}(1)|_{\Delta_{i}} )|_{X\times_{S} \mathcal{U}}\,$$
\fi
\end{proof}

\begin{lemma}\label{hom of certain sheaves}
 For $1\leq i,j\leq d$ and $i\neq j$ we have
$$\mathscr{H}om(\pi^{*}_{2,i}\mathcal{O}(1)|_{\Delta_{i}}, \pi^{*}_{2,j}\mathcal{O}(1)|_{\Delta_{j}})=0\,.$$
\end{lemma}

\begin{proof}
By adjunction, we have 
$${\ms Hom}(\pi^{*}_{2,i}\mathcal{O}(1)|_{\Delta_{i}}, \pi^{*}_{2,j}\mathcal{O}(1)|_{\Delta_{j}})={\ms Hom}(\pi^{*}_{2,i}\mathcal{O}(1)|_{\Delta_i\cap \Delta_j},\pi^{*}_{2,j}\mathcal{O}(1)|_{\Delta_{j}} )$$
Since $\Delta_{j}$ is an integral scheme and $\Delta_{i}\cap \Delta_{j}$ is a proper subset of $\Delta_{j}$, the later term in the above expression is zero.
\end{proof}
\begin{lemma}\label{comparision with global section of the tangent bundle of projective bundle} 
For any $1\leq j \leq d$ we have
$$H^{0}(X\times_{S} \mathcal{U},\mathscr{H}om ((\pi_{1}\times \pi_{2,j})^{*}\mathcal{F}(E,1),\pi_{2,j}^{*}\mathcal{O}(1)|_{\Delta_{j}}))= H^{0}(\mathbb{P}(E), \mathcal{T}_{\mathbb{P}(E)/S})\,.$$
\end{lemma}
\begin{proof}
The projection $\pi_2$ induces isomorphism 
$\Delta_{j}\xrightarrow{\sim} \mc Z$. 
Identifying $\Delta_j$ with $\mathcal{Z}$ we have
\begin{align*}
& H^{0}(X\times_{S} \mathcal{U}, \mathcal{H}om((\pi_{1}\times \pi_{2,j})^{*}\mathcal{F}(E,1),\pi^{*}_{2,j}\mathcal{O}(1)|_{\Delta_{j}}))\\
& =H^{0}(\mathcal{U},\mathcal{H}om (p^{*}_{j}\mathcal{F}(E,1)|_{\Delta_{1}},p^{*}_{j}\mathcal{O}(1) ))                                \\
& =H^{0}(\mathcal{U},p^{*}_{j}(\mathcal{F}(E,1)^{\vee}|_{\Delta_{1}}\otimes \mathcal{O}(1)))\,.
\end{align*}
Since $\mathcal{F}(E,1)$ is vector bundle 
over $\mathcal{Z}$ and codimension of 
$\mathcal{Z} \setminus \mathcal{U}\geq 2$ we have
$$H^{0}(\mathcal{U},p^{*}_{j}(\mathcal{F}(E,1)^{\vee}|_{\Delta_{1}}\otimes \mathcal{O}(1)))=H^{0}(\mathcal{Z},p^{*}_{j}(\mathcal{F}(E,1)^{\vee}|_{\Delta_{1}}\otimes \mathcal{O}(1)))\,.$$
Using projection formula for the morphism $p_{j}$ we get that
$$H^{0}(\mathcal{U},p^{*}_{j}(\mathcal{F}(E,1)^{\vee}|_{\Delta_{1}}\otimes \mathcal{O}(1)))=H^{0}(\mathbb{P}(E),(\mathcal{F}(E,1)^{\vee}|_{\Delta_{1}}\otimes \mathcal{O}(1)))\,.$$
Now over $\mathbb{P}(E)$ we have 
$$\mathcal{F}(E,1)^{\vee}|_{\Delta_{1}}\otimes \mathcal{O}(1)\cong \mathcal{T}_{\mathbb{P}(E)/S}\,.$$ 
This completes the proof of the lemma.
\end{proof}

\begin{lemma}\label{ext sheaves simplified}
For $1\leq i,j\leq d,~i\neq j$,
we have an isomorphism of sheaves:
$$\mathscr{E}xt^{1}(\pi^{*}_{2,i}\mathcal{O}(1)|_{\Delta_{i}}, \pi^{*}_{2,j}\mathcal{O}(1)|_{\Delta_{j}})\cong \pi^{*}_{2,i}\mathcal{O}(-1) \otimes \pi^{*}_{2,j}\mathcal{O}(1))\otimes \pi^{*}_{1}\mathcal{T}_{X/S}|_{\Delta_{i}\cap \Delta_{j}}\,.$$
\end{lemma}

\begin{proof}
Consider the exact sequence:
\[
\begin{tikzcd}
0 \arrow[r] & \mathcal{O}(-\Delta_{i}) \arrow[r] & \mathcal{O}_{X\times_{S}\mathcal{Z}} \arrow[r] & \mathcal{O}_{\Delta_{i}} \arrow[r] & 0
\end{tikzcd}
\]
Applying $\mathscr{H}om(\text{  }, \mathcal{O}_{\Delta_{j}})$ to the above exact sequence, we get:
\[
\begin{tikzcd}
0 \arrow[r] & \mathcal{O}_{\Delta_{j}} \arrow[r] & \mathcal{O}(\Delta_{i})|_{\Delta_{j}} \arrow[r] & \mathscr{E}xt^{1}(\mathcal{O}_{\Delta_{i}},\mathcal{O}_{\Delta_{j}}) \arrow[r] & 0
\end{tikzcd}
\]
Therefore 
$$\mathscr{E}xt^{1}(\mathcal{O}_{\Delta_{i}},\mathcal{O}_{\Delta_{j}})\cong \pi^{*}_{1}\mathcal{T}_{X/S}|_{\Delta_{i}\cap \Delta_{j}}\,. $$
and the statement follows immediately from this.
\end{proof}
The following corollary follows immediately from Lemma \ref{ext sheaves simplified}.
\begin{corollary}\label{pushforward of ext sheaves simplified}
We have an isomorphism of sheaves on $\mc Z$
$$(\pi_2)_*\mathscr{E}xt^{1}(\pi^{*}_{2,i}\mathcal{O}(1)|_{\Delta_{i}}, \pi^{*}_{2,j}\mathcal{O}(1)|_{\Delta_{j}}) \cong p^{*}_{i}\mathcal{O}(-1)\otimes p^{*}_{j}\mathcal{O}(1)\otimes (p\circ p_{i})^{*}\mathcal{T}_{X/S}\vert_{\Delta_{i,j,X}}\,.$$
\end{corollary}

\begin{lemma}\label{codimension of various closed sets} 
Fix $1\leq i,j\leq d$ with $i\neq j$. Then for $1\leq k,l,m\leq d$ with $k, l,m$ distinct we have  
\begin{enumerate}
\item  ${\rm codim}(\Delta_{k,l}\cap \Delta_{i,j,X},\Delta_{i,j,X})\geq 2$ if $\{k,l\}\neq \{i,j\}$.
\item ${\rm codim}(\Delta_{i,j}\cap ,\Delta_{i,j,X},\Delta_{i,j,X})=r$.
\item If $\{i,j\}\nsubseteq \{k,l,m\}$ then ${\rm codim}(\Delta_{k,l,m,X}\cap \Delta_{i,j,X},\Delta_{i,j,X})\geq 2$.
\item ${\rm codim}(\Delta_{i,j,k,X},\Delta_{i,j,X})=1$.
\end{enumerate}
\end{lemma}

\begin{proof} Without loss of generality we can assume $(i,j)=(1,2)$. Then 
$$\Delta_{1,2,X}\cong \mathbb{P}(E)^2_X\times_{S}\mathbb{P}(E)^{d-2}_{S}\,.$$ 
\begin{enumerate}
\item Let $\{k,l\}\cap \{1,2\}=\emptyset$. Without loss of generality we can assume 
$(k,l)=(3,4)$. Then 
$$\Delta_{k,l}\cap \Delta_{1,2,X} \cong \mathbb{P}(E)^2_X\times_{S}\mathbb{P}(E)\times_{S} \mathbb{P}(E)^{d-4}_{S}\,.$$
Let $\{k,l\}\cap \{1,2\}=\{2\}$. Without loss of generality we can assume 
$(k,l)=(2,3)$. Then
$$\Delta_{k,l}\cap \Delta_{1,2,X}\cong \mb P(E)^2_X \times_S \mb P(E)^{d-3}_S\,.$$
Therefore in both these cases it have codimension $\geq 2$ in $\Delta_{1,2,X}$.
\item If $\{k,l\}=\{1,2\}$ then $\Delta_{1,2}\cong\mathbb{P}(E)\times_{S}\mathbb{P}(E)^{d-2}_{S}$. Hence it has  codimension $r$ in $\Delta_{1,2,X}$.
\item Let  $\{1,2\}\cap \{k,l,m\}=\emptyset$. Without loss of generality we can assume 
$(k,l,m)=(3,4,5)$. Then 
$$\Delta_{k,l,m,X}\cap \Delta_{1,2,X}\cong \mathbb{P}(E)^2_X \times_{S} \mathbb{P}(E)^3_X\times_{S}\mathbb{P}(E)^{d-5}_{S}\,.$$
Let $\{i,j\}\cap \{k,l,m\}=\{i\}$. Without loss of generality we can assume $k=i=1$ and $(l,m)=(3,4)$. Then
 $$\Delta_{k,l,m,X}\cap \Delta_{1,2,X}\cong \mathbb{P}(E)^4_X \times_{S} \mathbb{P}(E)^{d-4}_{S}\,.$$
 Hence in both of these two cases it has codimension $\geq 2$ in $\Delta_{1,2,X}$.
\item $\Delta_{1,2,k,X}\cong \mathbb{P}(E)\times_{X} \mathbb{P}(E) \times_{X} \mathbb{P}(E)\times_{S} \mathbb{P}(E)^{d-3}_{S}$. Hence it has codimension $1$ in $Y$. 
\end{enumerate}
\end{proof}
On $\Delta_{i,j,X}$ we define the line bundle 
$$\mc L:=p^{*}_{i}\mathcal{O}(-1)\otimes p^{*}_{j}\mathcal{O}(1)\otimes (p\circ p_{i})^{*}\mathcal{T}_{X/S}\vert_{\Delta_{i,j,X}}\,.$$  
By Lemma \ref{codimension of various closed sets} we have that $\Delta_{i,j,k,X}\subset \Delta_{i,j,X}$ is a divisor on $\Delta_{i,j,X}$.
\begin{lemma} \label{extending sections-1}
Fix $1\leq i,j\leq d$ with $i\neq j$. For any $n\geq 0$ we have
$$H^{0}(\Delta_{i,j,X}, \mathcal{L}\otimes \mathcal{O}(\sum\limits_{k\neq i,j}n \cdot\Delta_{i,j,k,X}))=0\,.$$
\end{lemma}

\begin{proof} 
\if
Consider the following composition
$$\Delta_{i,j,X}\hookrightarrow \mc Z \to (\mb P(E))^{d-1}_S\,$$
where the second map is given by the product of projections $\prod\limits^d_{l=1,l\neq i}p_k$.
Let us denote this composition by $f$.
\fi
Let $f:\Delta_{i,j,X}\to \mb P(E)^{d-1}_S$ be the product of all the projections except the $i$-th projection. 
Then by projection formula
$$f_*(\mc L\otimes \mathcal{O}(\sum\limits_{k\neq i,j}n\cdot\Delta_{i,j,k,X}))=(f_* p^*_i\mc O(-1))\otimes \mc L' $$
for some line bundle $\mc L'$ on $\mb P(E)^{d-1}_S$. Consider  
the following fibered diagram:
\[
\begin{tikzcd}
\Delta_{i,j,X} \arrow[r, "p_{i}"] \arrow[d,"f"] & \mathbb{P}(E) \arrow[d,"p"] \\
\mathbb{P}(E)_{S}^{d-1} \arrow[r,"g_i"] & X 
\end{tikzcd}
\]
Here $g_i$ is the composition of $i$-th projection from $\mathbb{P}(E)_{S}^{d-1}$ and the morphism $p:\mb P(E)\to X$. Since $g_i$ is flat we have 
$$f_*p^*_i\mc O(-1)=g^*_ip_*\mc O(-1)\,.$$
Since $p_*\mc O(-1)=0$ we have that $f_*p^*_i\mc O(-1))=0$.  This completes the proof of the lemma.
\if
Hence it follows that $H^0(\Delta_{i,j,X},\mathcal{L}\otimes \mathcal{O}(\sum\limits_{k=1,k\neq i,j}^{d}n\Delta_{i,j,k,X}))=0$.

We will show that 
$$f_*(\mc L \otimes \mathcal{O}(\sum\limits_{k=3}^{d}n\Delta_{1,2,k,X})))=0\,.$$

\textbf{Claim.} $(p_{2}\times p_{3}\times ..\times p_{d})_{*}(\mathcal{L}\bigotimes \mathcal{O}(\sum\limits_{k=3}^{d}n\Delta_{1,2,k,X}))=0$\\

\textit{Proof of claim.}  Let us denote the $\mathbb{P}(E)\times_{S} (\mathbb{P}(E))_{S}^{d-2} \rightarrow \mathbb{P}(E)$ the $k$-th projection to $\mathbb{P}(E)$ by $p'_{k}$, where $3\leq k \leq d$.\\
Now, consider the following diagram:\\

\hspace{3cm}\begin{tikzcd}[row sep=huge, column sep=huge]
(\mathbb{P}(E) \times_{X} \mathbb{P}(E))\times_{S} (\mathbb{P}(E))_{S}^{d-2} \arrow[d," p_{2}\times p_{3}\times ..\times p_{d}=:f"] \\
(\mathbb{P}(E))\times_{S} (\mathbb{P}(E))_{S}^{d-2} \arrow[d,"(p\circ p_{2}')\times (p\circ p'_{3})\times ... \times(p\circ p'_{d})=: g"] \\
X\times_{S} (X)_{S}^{d-2}
\end{tikzcd}\\
Let us denote the map $p_{2}\times p_{3}\times ..\times p_{d}$ by $f$, and $(p\circ p'_{2})\times (p\circ p'_{3})\times ... \times (p\circ p'_{d}))$ by $g$. Then, if we denote by $X_{k}\subseteq X\times_{S} (X)_{S}^{d-2}$ the closed subscheme whose points are of the form $(x_{1},x_{2},..,x_{d-1})$ with $x_{1}=x_{k}$, then $\mathcal{O}(\sum\limits_{k=3}^{d}n\Delta_{1,2,k,X}))=(g\circ f)^{*}\mathcal{O}(\sum\limits^{d}_{k=2}nX_{k})$.\\
Hence, by projection formula we have\\
$
f_{*}(\mathcal{L}\bigotimes \mathcal{O}(\sum\limits_{k=3}^{d} n\Delta_{1,2,k,X}))=f_{*}((p_{1})_{*}\mathcal{O}(-1)\bigotimes(p_{2})^{*}\mathcal{O}(1)\bigotimes (p\circ p_{2})^{*}\mathcal{T}_{X/S}\bigotimes \mathcal{O}(\sum\limits_{k=3}^{d} n\Delta_{1,2,k,X})) \\
=(f_{*}((p_{1})^{*}\mathcal{O}(-1)))\bigotimes ((p'_{2})^{*}\mathcal{O}(1))\bigotimes (p\circ p'_{2})^{*}\mathcal{T}_{X/S}\bigotimes g^{*}\mathcal{O}(\sum\limits^{d}_{k=2}nX_{k}))
$
 \\
Hence, it is enough to show that $f_{*}(p_{1})^{*}\mathcal{O}(-1)=0$.\\
Now consider the following fibered diagram:\\

\hspace{2cm}\begin{tikzcd}
(\mathbb{P}(E)\times_{X}\mathbb{P}(E))\times_{S}(\mathbb{P}(E))_{S}^{d-2} \arrow[r, "p_{1}"] \arrow[d,"f"] & \mathbb{P}(E) \arrow[d,"p"] \\
\mathbb{P}(E)\times_{S} (\mathbb{P}(E))_{S}^{d-2} \arrow[r,"p\circ p'_{2}"] & X 
\end{tikzcd} \\

 Since  $p\circ p'_{2}$ is flat, we have by [7, Prop. 9.3] \\
 $f_{*}p^{*}_{1}\mathcal{O}(1)\cong (p\circ p'_{2})^{*}p_{*}\mathcal{O}(-1)=0$.\qed\\
\fi
\end{proof}

\begin{proposition} Let $r={\rm rank}~E\geq 3$.
For $1\leq i,j\leq d,~i\neq j$ we have
$$H^{0}(X\times_{S}\mathcal{U}, \mathscr{E}xt^{1}(\pi^{*}_{2,i}\mathcal{O}(1)|_{\Delta_{i}}, \pi^{*}_{2,j}\mathcal{O}(1)|_{\Delta_{j}}))=0\,.$$
\end{proposition}

\begin{proof}
By Corollary \ref{ext sheaves simplified} it is enough to show 
$$H^0(\mc U\cap \Delta_{i,j,X},\mc L)=0\,.$$
\if Let us denote the open set $\Delta_{i,j,X}\setminus (\bigcup\limits_{k\neq i,j} \Delta_{i,j,k,X})$ by $V$.\fi 
Since $r\geq 3$ by Lemma \ref{codimension of various closed sets}  we have  
$$H^{0}(\mathcal{U} \cap \Delta_{i,j,X}, \mathcal{L})=H^{0}(\Delta_{i,j,X}\setminus (\bigcup\limits_{k\neq i,j} \Delta_{i,j,k,X}),\mathcal{L})\,.$$
Let $s \in H^{0}(V,\mathcal{L}|_{V})$. Then for some $n$ large enough, there exists a section $0\neq t\in H^{0}(\Delta_{i,j,X},\mathcal{O}(\sum\limits_{k\neq i,j}n \cdot \Delta_{1,2,k,X}))$ such that 
the section $st^{n}$ extends to a global section of $\mathcal{L}\otimes \mathcal{O}(\sum\limits_{k\neq i,j}n\cdot\Delta_{1,2,k,X})$. However by Lemma \ref{extending sections-1} there are no global sections of this line bundle and this completes the proof of the proposition.
\end{proof}
Since $p_{S}$ is a projective morphism, we have a $p_{S}$-ample line bundle
$\mc O_{X}(1)$. Let $\mc O_{S}(1)$ be an ample line bundle on $S$. Then for
$a \gg 0$ the line bundle $\mc O_{X}(1)\otimes \mc O_{S}(a)$ is an ample line bundle on $X$. We fix such an ample line bundle $\mc M$ on $X$.
\begin{lemma}\label{extending section-2}
Let $E$ be semistable with respect to $\mc M$, ${\rm rank}~E=2$ and genus of $X_s\geq2$ for any $s\in S$.
Fix $1\leq i \leq d$. Then for any $n\geq 0$ we have
$$H^{0}(\mathcal{Z},p_{i}^{*}\mathcal{T}^{n}_{\mathbb{P}(E)/X}\otimes (p\circ p_{i})^{*}\mathcal{T}_{X/S} \otimes \mathcal{O}(\sum\limits_{k\neq i}n \cdot \Delta_{i,k,X}))=0\,.$$
\end{lemma}
\begin{proof} 
Without loss of generality we can assume $i=1$.
Let us denote the $j$-th projection from $(X)_{S}^{d}$ to $X$ by $p_{j,X}$. We define $X_{j}\subseteq (X)_{S}^{d}$ to be the closed set defined by the equation $p_{1,X}=p_{j,X}$. By projection formula we have
\begin{align*}
& (\prod\limits_{j}(p\circ p_j))_{*}(p_{i}^{*}\mathcal{T}^{n}_{\mathbb{P}(E)/X}\otimes (p\circ p_{i})^{*}\mathcal{T}_{X/S} \otimes \mathcal{O}(\sum\limits_{k\neq i}n \cdot \Delta_{i,k,X})) \\
& = p^{*}_{1,X}(S^{2n}(E)\otimes ({\rm det}(E^{\vee}))^{n})\otimes  p_{1.X}^{*} \mathcal{T}_{X/S} \otimes \mathcal{O}(\sum\limits_{k=2}^{d}n \cdot X_{k}) \,.
\end{align*}
Now we have the following exact sequence
\[
\begin{tikzcd}
0 \arrow[r] & \mathcal{O}(nX_{k}) \arrow[r] & \mathcal{O}((n+1)X_{k}) \arrow[r] & \mathcal{O}((n+1)X_{k})|_{X_{k}} \arrow[r] & 0 
\end{tikzcd}
\]
Note that $\mathcal{O}((n+1)X_{k})|_{X_{k}}=p^{*}_{1,X}\mathcal{T}^{n+1}_{X/S}$.
Tensoring the above exact sequence with $p^{*}_{1,X}(S^{2n}(E)\otimes (\text{det}(E^{\vee}))^{n})\otimes  p_{1.X}^{*} \mathcal{T}_{X/S}$ and applying $H^{0}$ we get that it is enough to show
$$H^{0}(\mathcal{Z},p^{*}_{1,X}(S^{2n}(E)\otimes (\text{det}(E^{\vee}))^{n})\otimes  p_{1.X}^{m} \mathcal{T}_{X/S})=0~\forall n\geq 0, m\geq 1\,.$$
Applying projection formula for the morphism $p_{1,X}$ we get  
\begin{align*}
& (p_{1,X})_{*}(p^{*}_{1,X}(S^{2n}(E)\otimes (\text{det}(E^{\vee}))^{n})\otimes  p_{1,X}^{m} \mathcal{T}_{X/S})                         \\
& = S^{2n}(E)\otimes (\text{det}(E^{\vee}))^{n} \otimes \mathcal{T}^{m}_{X/S} \,.
\end{align*}
Hence it is enough to show that 
$$H^{0}(X,S^{2n}(E)\otimes (\text{det}(E^{\vee}))^{n} \otimes \mathcal{T}^{m}_{X/S})=0~\forall n\geq 0, m\geq 1\,.$$
Now 
$${\rm deg}~S^{2n}(E)=\binom{2+2n-1}{2}{\rm deg}~E=n(2n+1){\rm deg}~E$$
and ${\rm rank} ~S^{2n}(E)=2n+1$.
Therefore  
$${\rm deg}~S^{2n}(E)\otimes (\text{det}(E^{\vee}))^{n} \otimes \mathcal{T}^{m}_{X/S}=m(2n+1)(\text{deg }\mathcal{T}_{X/S})\,.$$
Since genus of each fibre of $X\rightarrow S$ is $\geq 2$, deg $\mathcal{T}_{X/S}<0$. Hence
$${\rm deg}~S^{2n}(E)\otimes (\text{det}(E^{\vee}))^{n} \otimes \mathcal{T}^{m}_{X/S}<0\,.$$
Since $E$ is semistable we have that the bundle $S^{2n}(E)\otimes (\text{det}(E^{\vee}))^{n} \otimes \mathcal{T}^{m}_{X/S}$ is also semistable with negative degree. Therefore it does not have any global section.
\end{proof}
\begin{proposition}\label{global sections of ext sheaves}
Let $r={\rm rank}~E=2$, $E$ is semistable with respect to $\mc M$ and genus of $C$ is $\geq 2$. For $1\leq i,j\leq d,~i\neq j$ we have
$$H^{0}(X\times_{S}\mathcal{U}, \mathscr{E}xt^{1}(\pi^{*}_{2,i}\mathcal{O}(1)|_{\Delta_{i}}, \pi^{*}_{2,j}\mathcal{O}(1)|_{\Delta_{j}}))=0\,.$$
\end{proposition}

\begin{proof}
By Corollary \ref{ext sheaves simplified} it is enough to show 
$$H^0(\mc U\cap \Delta_{i,j,X},\mc L)=0\,.$$
Define the open set  
$$V:=\Delta_{i,j,X}\setminus (\Delta_{i,j} \bigcup \bigcup\limits^{d}_{k\neq i,j}\Delta_{i,j,k,X})\subset \Delta_{i,j,X}\,.$$ 
By Lemma \ref{codimension of various closed sets} we have 
$$H^{0}(\Delta_{i,j,X}\cap \mathcal{U},\mathcal{L})=H^{0}(V,\mathcal{L})\,.$$ Therefore, to show that this space vanishes, it is enough to show that 
$$H^{0}(Y,\mathcal{L}(n(\Delta_{i,j}+\sum\limits_{k\neq i,j}\Delta_{i,j,k,X})))=0~\forall n\geq 0\,.$$
Now consider the following exact sequence:
\[
\begin{tikzcd}
0 \arrow[r] & \mathcal{O}(n\cdot \Delta_{i,j}) \arrow[r] & \mathcal{O}((n+1)\cdot\Delta_{i,j})\arrow[r] & \mathcal{O}(n\cdot\Delta_{i,j})|_{\Delta_{i,j}} \arrow[r] & 0
\end{tikzcd}
\]
Tensoring the above exact sequence by $\mathcal{L}(n \cdot \sum\limits_{k\neq i,j}\Delta_{i,j,k,X})$ and applying $H^{0}$ we see that it is enough to show that 
$$H^{0}(\Delta_{i,j}, \mathcal{L}(n(\Delta_{i,j}+\sum\limits_{i=1}^{d}\Delta_{i,j,k,X}))|_{\Delta_{i,j}})=0\,.$$ 
Note that $\mathcal{O}(\Delta_{i,j})|_{\Delta_{i,j}}=p^{*}_{i}\mathcal{T}_{\mathbb{P}(E)/S}|_{\Delta_{i,j}}$. Then 
\begin{align*}
 & \mathcal{L}(n(\Delta_{i,j}+\sum\limits_{k\neq i,j}\Delta_{i,j,k,X}))|_{\Delta_{i,j}} \\
= &   p^{*}_{i}\mathcal{O}(-1) \otimes p^{*}_{j}\mathcal{O}(1)\otimes (p\circ p_{i})^{*}\mathcal{T}_{X/S} \otimes p^{*}_{i}\mathcal{T}^{n}_{\mathbb{P}(E)/X} \otimes \mathcal{O}(\sum\limits_{k\neq i,j}n \cdot \Delta_{i,j,k,X}))|_{\Delta_{i,j}} \\
= &    p^{*}_{i}\mathcal{T}^{n}_{\mathbb{P}(E)/X} \otimes (p\circ p_{i})^{*}\mathcal{T}_{X/S} \otimes \mathcal{O}(\sum\limits_{k\neq i,j}n \cdot \Delta_{i,j,k,X}))|_{\Delta_{i,j}}\,.
\end{align*}
Identifying $\Delta_{i,j}$ with $\mathbb{P}(E)_{S}^{d-1}$ the statement follows from 
Lemma \ref{extending section-2}.
\end{proof}
\begin{proposition}\label{global sections of pullback of the tangent bundle}
Suppose either $r\geq 3$ or $r=2$, $E$ is semistable with respect to $\mc M$ and genus of $C$ is $\geq 2$. In both of these cases we have  an isomorphism
$$H^{0}(\mathcal{U},\Phi^{*}\mathcal{T}_{\mathcal{Q}(E,d)/S})=\bigoplus\limits_{i=1}^{d}H^{0}(\mathbb{P}(E),\mathcal{T}_{\mathbb{P}(E)/S})\,.$$
\end{proposition}
\begin{proof}
By Lemma \ref{pullback of the tangent bundle} we have
$$\Phi^*\mc T_{\mc Q(E,d)/S}\cong \bigoplus\limits^d_{i=1}(\pi_2)_*\ms Hom(\mc F(E,d),\pi^{*}_{2,i}\mathcal{O}(1)|_{\Delta_{i}})\vert_{\mc U}\,.$$
Hence for a fixed $1\leq j\leq d$ it is enough to show 
$$H^0(X\times_S \mc U,\ms Hom(\mc F(E,d),\pi^{*}_{2,j}\mathcal{O}(1)|_{\Delta_{j}})=H^{0}(\mathbb{P}(E),\mathcal{T}_{\mathbb{P}(E)/S})\,.$$
Over $X\times_{S} \mathcal{U}$ we have the following commutative diagram:
\[
\begin{tikzcd} 
0 \arrow[r] & \mathcal{F}(E,d) \arrow[d] \arrow[r]                    & \pi^{*}_{1}E \arrow[r] \arrow[d, "\cong"] & \bigoplus\limits_{i=1}^{d} \pi^{*}_{2,i} \mathcal{O}(1)|_{\Delta_{i}} \arrow[r] \arrow[d]                                                  & 0 \\
0 \arrow[r] & (\pi_{1}\times \pi_{2,j})^{*}\mathcal{F}(E,1) \arrow[r] & \pi^{*}_{1}E \arrow[r]                    & \pi^{*}_{2,j}\mathcal{O}(1)|_{\Delta_{j}} \arrow[r] & 0 \\
\end{tikzcd}
\]
Using snake lemma for the above diagram we get the following exact sequence over $X\times_{S} \mathcal{U}$
$$0 \to \mathcal{F}(E,d) \to (\pi_{1} \times \pi_{2,j})^{*}\mathcal{F}(E,1) \to \bigoplus_{i=1,i\neq j}^d \pi^{*}_{2,i}\mathcal{O}(1)|_{\Delta_{i}} \to 0$$
\if
By Lemma 
\begin{enumerate}
\item $\mathscr{H}om(\pi^{*}_{2,i}\mathcal{O}(1)|_{\Delta_{i}}, \pi^{*}_{2,j}\mathcal{O}(1)|_{\Delta_{j}})=0\,.$
\item $H^{0}(X\times_{S} \mathcal{U},\mathscr{H}om ((\pi_{1}\times \pi_{2,j})^{*}\mathcal{F}(E,1),\pi_{2,j}^{*}\mathcal{O}(1)|_{\Delta_{j}}))= H^{0}(\mathbb{P}(E), \mathcal{T}_{\mathbb{P}(E)/S})\,.$
\item $H^{0}(X\times_{S}\mathcal{U}, \mathscr{E}xt^{1}(\pi^{*}_{2,i}\mathcal{O}(1)|_{\Delta_{i}}, \pi^{*}_{2,j}\mathcal{O}(1)|_{\Delta_{j}}))=0\,.$
\end{enumerate}
\fi
We apply $\mathscr{H}om($  $ ,\pi^{*}_{2,j}\mathcal{O}(1)|_{\Delta_{j}} )$ and then the 
the functor $H^0$. Now the result follows from Lemma \ref{hom of certain sheaves}, Lemma \ref{comparision with global section of the tangent bundle of projective bundle} and Proposition \ref{global sections of ext sheaves}.
\end{proof}
\if
 we get the Lemma 2.4(i) to the exact sequence $(2.3)$, we get that over $X\times_{S} \mathcal{U}$, we have the following exact sequence:
\begin{align*}
0 \to \mathcal{H}om((\pi_{1}\times \pi_{2,j})^{*}\mathcal{F}(E,1),\pi_{2,j}^{*}\mathcal{O}(1)|_{\Delta_{j}}) \to & \mathcal{H}om(\mathcal{F}(E,d),\pi_{2,j}^{*}\mathcal{O}(1)|_{\Delta_{j}})  \\
\to & \mathcal{E}xt^{1}(\bigoplus\limits_{i=1,i\neq j}^{d}\pi^{*}_{2,i}\mathcal{O}(1)|_{\Delta_{i}},\pi_{2,j}^{*}\mathcal{O}(1)|_{\Delta_{j}}) \to 0
\end{align*}

\begin{tikzcd}
0 \arrow[r] & \mathcal{H}om((\pi_{1}\times \pi_{2,j})^{*}\mathcal{F}(E,1),\pi_{2,j}^{*}\mathcal{O}(1)|_{\Delta_{j}}) \arrow[r] & \mathcal{H}om(\mathcal{F}(E,d),\pi_{2,j}^{*}\mathcal{O}(1)|_{\Delta_{j}}) \arrow[d]\\
   & 0  & \mathcal{E}xt^{1}(\bigoplus\limits_{i=1,i\neq j}^{d}\pi^{*}_{2,i}\mathcal{O}(1)|_{\Delta_{i}},\pi_{2,j}^{*}\mathcal{O}(1)|_{\Delta_{j}}) \arrow[l]
\end{tikzcd}\\
Applying $H^{0}$ to the above exact sequence and using Lemma 2.4(i) and (ii), we get that \\
$H^{0}(X\times_{S} \mathcal{U}, \mathcal{H}om(\mathcal{F}(E,d),\pi_{2,j}^{*}\mathcal{O}(1)|_{\Delta_{j}})=H^{0}(\mathbb{P}(E),\mathcal{T}_{\mathbb{P}(E)/S})$. 

\fi

\begin{theorem}\label{main theorem} 
Suppose either $r:={\rm rank}~E\geq 3$ or $r=2$, $E$ is semistable with respect to $\mc M$ and genus of $X_s$ is $\geq 2$ for $s\in S$. In both of these cases we have isomorphisms
\begin{enumerate}
\item ${\rm Aut}^{o}(\mathcal{Q}(E,d))\cong {\rm Aut}^{o}(\mathbb{P}(E)/S)$.
\item $H^{0}(\mathbb{P}(E),\mathcal{T}_{\mathbb{P}(E)/S})\cong H^{0}(\mathcal{Q}(E,d),\mathcal{T}_{\mathcal{Q}(E,d)/S})$.
\end{enumerate}
\end{theorem}

\begin{proof}
By  Corollary \ref{inclusion of lie algebras} we have an inclusion of lie algebras: 
\begin{equation}\label{eqn-inclusion of lie alg}
H^0(\mb P(E),\mc T_{\mb P(E)/S})\hookrightarrow H^0(\mc Q(E,d),\mc T_{\mc Q(E,d)/S})\,.
\end{equation}
By Lemma \ref{pullback of the tangent bundle} and Proposition \ref{global sections of pullback of the tangent bundle} we have an inclusion 
$$H^0(\mc Q(E,d),\mc T_{\mc Q(E,d)/S})\hookrightarrow \bigoplus\limits^d_{i=1}H^0(\mb P(E),\mc T_{\mb P(E)/S})\,.$$
Since $\mc Z\to \mc Q(E,d)$ is invariant under the action of the symmetric group 
$S_d$ we get that this inclusion factors through 
$$H^0(\mc Q(E,d),\mc T_{\mc Q(E,d)/S})\hookrightarrow(\bigoplus\limits^d_{i=1}H^0(\mb P(E),\mc T_{\mb P(E)/S}))^{S_d}=H^0(\mb P(E),\mc T_{\mb P(E)/S})\,.$$ 
Comparing the dimensions we get that the (\ref{eqn-inclusion of lie alg}) is an isomorphism. Hence the inclusion in Lemma \ref{inclusion of aut of quot in aut of proj} is an isomorphism.
\end{proof}

\section{Applications}\label{Applications}
\begin{corollary} \label{main corollary}
Suppose either $r\geq 3$ or $r=2$, $E$ is semistable with respect to $\mc M$ and genus of $C$ is $\geq 2$. Then we have the following  left exact sequence of algebraic groups
$$0 \to {\rm GL}(E)/k^* \to {\rm Aut}^o(\mc Q(E,d)/S)\to {\rm Aut}^o(X/S)$$
The corresponding sequence of lie algebras is given by
$$0 \to H^0(X,{\rm ad}~E) \to H^0(\mc Q(E,d),\mc T_{\mc Q(E,d)/S})\to H^0(X,\mc T_{X/S})$$
\end{corollary}
\begin{proof}
The left exactness of the above sequences follow from Theorem \ref{main theorem} and from the fact that $\text{Aut}^{o}(\mathbb{P}(E)/S)$ and its lie algebra fits into the above exact sequences.
\end{proof}

\begin{corollary}\label{corollary-g greater than 1}
Let the genus of the fibres of $X\rightarrow S$ is $\geq 2$. Suppose either $r\geq 3$ or $r=2$ and $E$ is semistable with respect to $\mc M$. Then 
\begin{enumerate}
\item ${\rm Aut}^{o}(\mathcal{Q}(E,d)/S)=\text{GL}(E)/k^{*}$.
\item $H^{0}(\mathcal{Q}(E,d),\mathcal{T}_{\mathcal{Q}(E,d)/S})=H^{0}(X,\text{ad }E)$.
\end{enumerate}
\end{corollary}
\begin{proof}
If genus of each fibre is $\geq 2$ then $(p_{S})_{*}\mathcal{T}_{X/S}=0$. In particular $H^{0}(X,\mathcal{T}_{X/S})=0$. Hence $\text{Aut}^{o}(X/S)=0$. Now the corollary follows from Corollary \ref{main corollary}.
\end{proof}
Taking $S$ to be a point and $E=\mc O^r_C$ in Corollary \ref{corollary-g greater than 1} we get  \cite[Theorem 3.1]{BDH} and \cite[Corollary 3.2]{BDH}.  
\begin{corollary} 
Let $C$ be a smooth projective curve of genus $\geq 2$ over an algebraically closed field $k$ of characteristic zero. Then
\begin{enumerate}
\item ${\rm Aut}^{o}(\mathcal{Q}(\mc O^r_C,d)/S)=\text{PGL}(r)$.
\item $H^{0}(\mathcal{Q}(\mc O^r_C,d),\mathcal{T}_{\mathcal{Q}(\mc O^r_C,d)})=\mf sl(r)$.
\end{enumerate}
\end{corollary}
Let $C$ be a smooth projective curve of genus $\geq 2$ over an algebraically closed field $k$ of characteristic zero. Fix $\underline{d}=(d_{1},d_{2},\ldots,d_{k})\in \mathbb{N}^{k}$ with $d_{1}>d_{2}>\ldots>d_{k}$ and $r\geq 1$. Let $\mathcal{D}(r,\underline{d})$ be the flag scheme parametrizing chain of quotients of $\mathcal{O}^{r}_{C}\to B_{1}\to B_{2}\to \ldots \to B_{d}$ where $B_{i}$ is a torsion quotient of degree $d_{i}$ \cite[2.A.1]{HL}. It is known that $\mc D(r,\underline{d})$ is a smooth projective variety.
\begin{corollary}\label{automorphisms of flag schemes}
We have the following isomorphisms of algebraic groups and lie algebras
\begin{enumerate}
\item ${\rm Aut}^{o}(\mathcal{D}(r,\underline{d}))\cong {\rm PGL}(r)$.
\item $H^{0}(\mathcal{D}(r,\underline{d}), \mathcal{T}_{D(r,\underline{d})})= \mf{sl}(r)$.
\end{enumerate}
\end{corollary}
\begin{proof}
Let $\underline{d}':=(d_2,d_3,\ldots,d_k)$.
Over $C\times \mathcal{D}(r,\underline{d}')$ we have the universal chain of filtrations:
$$
\mathcal{A}(r,d_{2})\subset \mathcal{A}(r,d_{3})\subset \ldots \subset \mathcal{A}(r,d_{k})\subset \mathcal{O}^{r}_{C\times \mathcal{D}(r,\underline{d}')}\,.
$$
Then  $D(r,\underline{d})$ is the relative quot scheme of torsion quotients of degree $d_{1}-d_{2}$ of the vector bundle $\mathcal{A}(r,d_{2})$ for the map 
$$C\times \mathcal{D}(r,\underline{d}')\rightarrow \mathcal{D}(r,\underline{d}')\,.$$ 
By Corollary \ref{corollary-g greater than 1} we get that
$$H^{0}(\mathcal{D}(r,\underline{d}),\mathcal{T}_{\mathcal{D}(r,\underline{d})/\mathcal{D}(r,\underline{d}')})=H^{0}(C\times \mathcal{D}(r,\underline{d}'), {\rm ad }~\mathcal{A}(r,d_{2}))\,.$$ 
By \cite[Theorem 3.2.4, Theorem 5.1]{G18} the bundle $\mathcal{A}(r,d_{2})$ is stable with respect to certain polarisations on $C\times \mathcal{D}(r,\underline{d}')$. Hence by Corollary \ref{corollary-g greater than 1} 
we have 
$$H^{0}(C\times \mathcal{D}(r,\underline{d}'), \text{ad }\mathcal{A}(r,d_{2}))=0\,.$$ 
By induction on k we get that 
$$H^{0}(\mathcal{D}(r,\underline{d}),\mathcal{T}_{D(r,\underline{d})})=H^{0}(C,\text{ad }\mathcal{O}^{r}_{C})=\mf{sl}(r)\,.$$ 
This completes the proof of the corollary.
\end{proof}
Let $C$ be a smooth projective curve over an algebraically closed field of characteristic zero. In \cite{BM} the authors computed the identity component of automorphism group scheme
of a certain generalized quot scheme $\mc Q_C(r,d_p,d_z)$. We recall the definition of this scheme: Fix $r\geq 2,d_p,d_z\geq 1$. Consider the quot scheme $\mc Q(\mc O^r_C,d_p)$ and the universal kernel bundle $\mc A(r,d_p)$  over $C\times \mc Q(\mc O^r_C,d_p)$. Then $\mc Q(r,d_p,d_z)$ is defined as the relative Quot scheme associated to the projection $C\times \mc Q(\mc O^r_C,d_p)\to \mc Q(\mc O^r_C,d_p)$ and the bundle $\mc A(r,d_p)^{\vee}$. By \cite[Theorem 3.2.4]{G18} $\mc A(r,d_p)$ is stable with respect to certain polarisations. Hence 
$H^0(C\times \mc Q(r,d_p),{\rm ad}~\mc A(r,d_p)^{\vee})=0$. Now from Theorem 
\ref{main theorem} we get the result proved in \cite{BM}:
\begin{corollary}\cite[Theorem 2.1]{BM} \label{automorphism of generalised quot scheme}
Let $C$ be a smooth projective curve of genus $\geq 2$ over an algebraically closed field $k$ of characteristic zero. We have the following isomorphisms of algebraic groups and lie algebras
\begin{enumerate}
\item ${\rm Aut}^{o}(\mathcal{Q}(r,d_p,d_z))\cong {\rm PGL}(r)$.
\item $H^{0}(\mathcal{Q}(r,d_p,d_z), \mathcal{T}_{\mc Q(r,d_p,d_z)})= \mf{sl}(r)$.
\end{enumerate}
\end{corollary}

\begin{corollary}
Let $C$ be a smooth projective curve over an algebraically closed field $k$. Let $E$ be a vector bundle of rank $\geq 3$ over $C$. Fix $d\geq 1$. Let $\mathcal{Q}(E,d)$ be the quot scheme of torsion quotients of $E$ of degree $d$. Then we have
\begin{enumerate}
\item\label{projective line} If genus of $C=0$,i.e. $C\cong \mathbb{P}^{1}$, then
$${\rm Aut}^{o}(\mathcal{Q}(E,d))={\rm PGL}(2,k)\ltimes {\rm GL}(E)/k^{*}\,.$$
\item\label{semistable bundles-elliptic curves} If genus of $C=1$ and if $E$ is semistable then 
we have the following sequence of algebraic groups
$$0 \to {\rm GL}(E)/k^*\to {\rm Aut}^o(\mc Q(E,d))\to {\rm Aut}^o(C)\to 0\,.$$
\item\label{unstable bundles-elliptic curves} If $E$ is not semistable, then ${\rm Aut}^{o}(\mathcal{Q}(E,d))=\text{GL}(E)/k^{*}$.
\end{enumerate}

\end{corollary}
\begin{proof}
If $C\cong \mathbb{P}^{1}$ then any vector bundle $E$ admits a $\text{GL}(2)$ linearisation, in paricular we have a homomorphism GL$(2)\rightarrow \text{Aut}^{o}(\mathcal{Q}(E,d))$. This homomorphism factors through $\text{PGL}(2,k)$ and gives a section to the map $\text{Aut}^{o}(\mathcal{Q}(E,d))\rightarrow \text{PGL}(2,k)$. Therefore the left exact sequence in Corollary \ref{main corollary} is exact in this case and it splits.

From now on we assume that genus of $C$ is $1$ i.e. $C$ is an elliptic curve. Recall that a bundle $E$ is called semi-homogeneous if ${\rm Aut}^o(\mb P(E))\to {\rm Aut}^o(C)=C$ is surjective (\cite[Definition 5.2]{Mu}). By \cite[Proposition 6.13]{Mu} every semi-homogenous bundle is semistable. Hence (\ref{unstable bundles-elliptic curves}) follows from Corollary \ref{main corollary}. Let us assume $E$ is semistable. Then $E\cong \oplus E_{i}$, where $E_i$ are indecomposable of slope $\mu(E_i)=\mu(E)$. By \cite[Theorem 10]{At} any indecomposable bundle over $C$ is semi-homogenous and therefore by \cite[Proposition 6.9]{Mu} we have that $E$ is semi-homogenous. Now (\ref{semistable bundles-elliptic curves}) follows from Corollary \ref{main corollary}. 
\if
then  Let $E$ be a semistable bundle on $C$. 
We   We show that for any $g\in \text{Aut}^{o}(C)$ $g^{*}E\cong E$. This will show that the sequence (3.1) is also right exact in this case.\\ 
We know that  $E\cong \bigoplus E_{i}$, where $E_{i}$'s are indecomposable vector bundles. Since, $E$ is semistable, hence $\mu(E_{i})=\mu(E_{j})$ $\forall i,j$. Since the $\mathcal{Q}(E,d)\cong \mathcal{Q}(E\bigotimes \mathcal{L},d)$ canonically for any line bundle $\mathcal{L}$, after twisting $E$ by a line bundle of appropriate degree, we can assume $\mu(E)=\mu(E_{i})=0$ $\forall i$. By [5,Th.5(i)], $E_{i}\cong F_{r_{i}}\bigotimes \mathcal{L}_{i}$, where $F_{r_{i}}$ is the unique indecomposable vector bundle of rank $r_{i}$ with $H^{0}(C,E_{r_{i}})\neq 0$ and $\mathcal{L}$ is a line bundle of degree $0$.\\
It follows that for any $g\in \text{Aut}^{o}(C)$, $g^{*}F_{r}\cong F_{r}$, since $g^{*}F_{r}$ is also indecomposable bundle of degree $0$ and rank $r$ with $H^{0}(C,g^{*}F_{r})\neq 0$.\\
So, we need to show that $g^{*}\mathcal{L}\cong \mathcal{L}$ for any $\mathcal{L}\in \text{Pic}^{o}(C)$.\\
Fix a base point $x_{0}\in C$. Then, under the group structure of $C$ with $x_{0}$ as the identity, $\text{Aut}^{o}(C)\cong C$ i.e. any element of $\text{Aut}^{o}(C)$ is given by $y \mapsto y+_{C}x$ for a fixed $x\in C$.(Here we denote the group addition of $C$ by $+_{C}$) Fix such an automorphism $x\in C\cong \text{Aut}^{0}(C)$\\
Now, $(C,x_{0},+_{C})\cong \text{Pic}^{o}(C)$, with the morphism given by $z \mapsto \mathcal{O}(z-x_{0})$.\\
Let us assume $\mathcal{L}=\mathcal{O}(z-x_{0})$ for $z\in C$.\\
Then $x^{*}\mathcal{L}=\mathcal{O}((z-_{C}x)-(x_{0}-_{C}x))$.\\
Since, $(C,x_{o},+_{C})\cong \text{Pic}^{0}(C)$ is a homomorphism, it follows that \\$x^{*}\mathcal{L}\cong \mathcal{O}((z-_{C}x)-(x_{0}-_{C}x))=\mathcal{O}(((z-_{C}x)-x_{0})-((x_{0}-_{C}x)-x_{0}))\\
=\mathcal{O}((z-x_{0})-(x-x_{0})-(x_{0}-x_{0})+(x-x_{0}))=\mathcal{O}(z-x_{0})=\mathcal{L}$.\\
Hence, it follows that for any $g\in \text{Aut}^{o}(C)$, $g^{*}E\cong E$ for any $E$ semistable.\\
(ii) By [6, Proposition 6.13], every semi-homogenous vector bundle [6, Definition 5.2] is semistable. In particular, if $E$ is not semistable, then the map $H^{0}(\mathbb{P}(E),\mathcal{T}_{\mathbb{P}(E)})\rightarrow H^{0}(C,\mathcal{T}_{C})$ is zero. Hence, using sequence (3.2), we get $H^{0}(C,\text{ad }E)\rightarrow H^{0}(\mathcal{Q}(E,d),\mathcal{T}_{\mathcal{Q}(E,d)})$ is an isomorphism. From this, it follows that $\text{Aut}^{o}(\mathcal{Q}(E,d))=\text{GL}(E)/k^{*}$.
 \fi
\end{proof}


\end{document}